\documentclass{amsart}
\usepackage[english]{babel}
\usepackage{amsmath,amsfonts,amssymb,amsthm}
\usepackage[all]{xy}
\usepackage[mathcal]{eucal}
\usepackage{mathrsfs}
\usepackage{caption}
\usepackage{hyperref}
\newtheorem{lm}{Lemma}[section]
\newtheorem{prop}[lm]{Proposition}
\newtheorem{thm}[lm]{Theorem}
\newtheorem{cor}[lm]{Corollary}
\newtheorem{df}[lm]{Definition}
\theoremstyle{remark}
\newtheorem{claim}[lm]{Claim}
\theoremstyle{definition}
\newtheorem{rk}[lm]{Remark}
\newtheorem{ex}[lm]{Example}

\renewcommand{\labelenumi}{(\alph{enumi})}

\newcommand{\C}{\mathbb{C}}
\newcommand{\R}{\mathbb{R}}
\newcommand{\val}{\mathrm{val}}
\newcommand{\Idlsh}{\mathscr{I}}
\newcommand{\regsh}{\mathcal{O}}
\newcommand{\dual}{{}^\vee}
\newcommand{\Q}{\mathbb{Q}}
\newcommand{\Z}{\mathbb{Z}}
\renewcommand{\div}{\mathrm{div}}
\newcommand{\Lscr}{\mathscr{L}}
\newcommand{\Fscr}{\mathscr{F}}
\newcommand{\Kscr}{\mathscr{K}}
\newcommand{\ord}{\mathrm{ord}}
\newcommand{\Pspace}{\boldsymbol{\mathrm{P}}}

\newcommand{\Aspace}{\boldsymbol{\mathrm{A}}}

\newcommand{\discrep}{\mathrm{discrep}}
\newcommand{\desc}[1]{\textnormal{#1}}
\newcommand{\Pic}{\mathrm{Pic}}
\newcommand{\e}{\varepsilon}
\newcommand{\Rscr}{\mathscr{R}}
\newcommand{\Escr}{\mathscr{E}}
\newcommand{\dfrak}{\mathfrak{d}}
\newcommand{\exc}{\mathrm{exc}}
\newcommand{\supp}{\mathrm{supp}}
\newcommand{\Jscr}{\mathscr{J}}
\newcommand{\ltp}{lt${}^+$ }
\newcommand{\relSpec}{\boldsymbol{\mathrm{Spec}}}

\title{Some properties and examples of log terminal${}^+$ singularities}
\author{alberto chiecchio}
\date{}
\begin{document}
\begin{abstract} In \cite{dFH}, de Fernex and Hacon started the study of singularities on non-$\Q$-Gorenstein varieties using pullbacks of Weil divisors. In \cite{ltsings}, the author of this paper and Urbinati introduce a new class of singularities, called \emph{log terminal${}^+$}, or simply \emph{\ltp}, which they prove is rather well behaved. In this paper we will continue the study of \ltp singularities, and we will show that they can be detected by a multiplier ideal, that they satisfy a Bertini type result, inversion of adjunction and small deformation invariance, and that they are naturally related to rational singularities. Finally, we will provide a list of examples (all of them with \ltp singularities) of the pathologies that can occur in the study of non-$\Q$-Gorenstein singularities.
\end{abstract}
\maketitle

\tableofcontents

\section{Introduction}
In \cite{dFH}, de Fernex and Hacon started the study of singularities on non-$\Q$-Gorenstein varieties using pullbacks of Weil divisors. In this work they introduce a notion of log canonical, log terminal, canonical and terminal singularities in this context. In \cite{stefano}, the author gives some first examples of the possible pathologies that can arise in this context, for example a variety with canonical but not log terminal singularities (\cite{stefano}, 4.1, example \ref{ex:lt+,notCM} below). In \cite{stefano2}, the author proceed with his study of non-$\Q$-Gorenstein canonical singularities, and their relation with divisorial models. In \cite{ltsings}, the author of this paper and Urbinati introduce a new class of singularities, called \emph{log terminal${}^+$}, or simply \emph{\ltp}, which they prove is rather well behaved. For example, the canonical algebra $\Rscr(X,K_X)$ is finitely generated in this case (\cite{ltsings}, 5.10, theorem \ref{thm:fgcr}).

In this paper we will continue the study of \ltp singularities, and we will show that they there exists a multiplier ideal detecting them (corollary \ref{cor:multideal}), that they satisfy a Bertini type result (theorem \ref{thm:bertini}), inversion of adjunction (theorem \ref{thm:invofadjunction}) and small deformation invariance (corollary \ref{cor:definv}), and that they are naturally related to rational singularities (theorem \ref{thm:lt++CM=>rtl}). Then we will focus on examples and their meaning. We will provide new examples of possible pathologies for the study of singularities in the non-$\Q$-Gorenstein case. Interestingly enough, all the examples that we provide are of singularities which are \ltp. The main point is that these pathologies occur even for singularities that are very well behaved. Moreover, as it is argued in \cite{ltsings}, \ltp singularities seems to be the largest class of singularities, at the moment, where it is possible to run a ``non-$\Q$-Gorenstein MMP without boundaries''. Therefore, the examples presented here are to be considered as a cautioning collection for everyone interested in the project.

\vspace{2ex}

We will recall the definition of the main objects, and some recent results in section 2. We will also notice that, under some restrictive conditions, the restriction of a pullback of a Weil divisor to a fiber is always numerically antieffective (corollary \ref{cor:pullback and restriction}). The general situation seems to be more complicated, and it is briefly discussed in remark \ref{rk:positivity on fibers}.

In section 3 we will prove several properties of \ltp singularities. We will show the existence of a multiplier ideal $\Jscr^+(X)$ (definition \ref{df:multideal} and lemma \ref{lm:multideal}) which detects \ltp singularities.
\begin{thm}[Corollary \ref{cor:multideal}] Let $X$ be a normal variety. Then $X$ has \ltp singularities if and only if $\Jscr^+(X)=\regsh_X$.
\end{thm}
Moreover, we will prove a Bertini type theorem.
\begin{thm}[Theorem \ref{thm:bertini}] Let $X$ be a normal variety having \ltp singularities. Then the generic hyperplane section of $X$ has \ltp singularities.
\end{thm}
We will also prove inverse of adjunction.
\begin{thm}[Theorem \ref{thm:invofadjunction}] Let $X$ be a normal variety, and let $S$ be an effective (normal and reduced) Cartier divisor in $X$ having \ltp singularities. Then $X$ has \ltp singularities in a neighborhood of $S$.
\end{thm}
As immediate consequence of these two results we obtain the invariance under small deformations.
\begin{cor}[Corollary \ref{cor:definv}] Let $f:X\rightarrow T$ be a proper flat family of varieties over a smooth curve $T$ and $t\in T$ a closed point. If the fiber $X_t$ has \ltp singularities, then so do the other fibers near $t$.
\end{cor}
Finally we will prove the following relation between \ltp singularities and rational singularities.
\begin{thm}[Theorem \ref{thm:lt++CM=>rtl}] If $X$ has Cohen-Macaulay and \ltp singularities, then it has rational singularities.
\end{thm}

\vspace{2ex} 

Sections 4 and 5 will be mainly justified by the following observation. If $X$ is a $\Q$-Gorenstein variety, $f:Y\rightarrow X$ is a resolution, $E$ is one of the components of the exceptional divisor which is mapped to a point, there is a standard technique to compute the discrepancy (or the log discrepancy in the more general setting). For example, in the case that $X$ is a cone over some smooth polarized variety $(V,L)$, and $f:Y\rightarrow X$ is the blow-up of the vertex, we can write $K_Y=f^*(K_X)+A$ (as $\Q$-divisors), where $A=aE$ is the discrepancy. Thus, $K_Y+E=f^*(K_X)+aE+E$. Restricting this identity to $E$, $(K_Y+E)\big|_E\sim K_E$ and $f^*(K_X)\big|_E\sim_\Q0$ ($E$ is a fiber), and we find $K_E\sim_\Q(a+1)E\big|_E\sim_\Q-(a+1)L$. We can use this to determine $a$, and thus $A=aE$. This procedure can fail in several steps if we use the pullback of \cite{dFH}.

Let $D$ be a Weil divisor on a normal variety $X$, and $F$ be the fiber of a birational morphism $f:Y\rightarrow X$. Recall that the pullback is defined as a limit $f^*(D)=\inf f^\natural(mD)/m$. If $D$ is not $\Q$-Cartier, and $F$ is a fiber of a (birational) morphism $f:Y\rightarrow X$, in general we will not be able to restrict $f^*(D)$ to $F$. Even if we can (for example if $F$ is a normal reduced Cartier divisor) there is no reason to suspect that $f^*(D)\big|_F\sim_\Q0$.
{\renewcommand{\labelenumi}{(F\arabic{enumi})}
\begin{enumerate}
\item\label{it:F1} If each $f^\natural(mD)/m\big|_F\sim_\Q\Gamma_m$, for some divisor $\Gamma_m$, it does not make sense to talk about an infimum of the $\Gamma_m$'s, as they are defined as $\Q$-linear classes (see remark \ref{rk:infimum of divisors}). So, we cannot compute the restriction of the pullback as a limit. We point out, that it will not be, in general, $f^\natural(mD)/m\big|_F\sim_\Q0$.
\item\label{it:F2} We should not expect that, in general, $f^*(D)\big|_F\equiv 0$ or $f^*(D)\big|_F\sim_\R0$.
\item\label{it:F3} Since we are working with a limiting process, and linear equivalence is not preserved in the limiting process, we should expect that $f^*(D)\big|_F\sim_\R\Gamma$ or $f^*(D)\big|_F\equiv\Gamma$ will encode very different meanings.
\end{enumerate}}
We will discuss each of this conditions, providing examples.

Another approach to the non-$\Q$-Gorenstein case is to use boundaries, and work in a log setting. This approach is the most commonly used, and in \cite{dFH} the authors relate some of the singularities they are defining to the log ones (\cite{dFH}, 7.2). The basic idea is to compare a discrepancy of $X$ for a sufficiently high resolution with the discrepancy of a pair $(X,\Delta)$ for a suitable boundary $\Delta$. This approach seems to fail if the resolution is given, and we will also study this case.

\vspace{2ex}

Section 4 will be focused on examples. 

To be more emphatic on how careful we have to be when dealing with \ltp singularities, we will show an example (example \ref{ex:Km>-1,K=-1}) where $K^+_{m,Y/X}>-1$ for each $m$, but $K^+_{Y/X}=-1$, so that the multiplier ideal in definition \ref{df:multideal} had to be constructed by directly looking at $K^+_{Y/X}$ and not by doing a limiting process on the $K^+_{m,Y/X}$ (see remark \ref{rk:multideal}).

We will show an example (example \ref{con:example}) of a cone $X$, where, if $f:Y\rightarrow X$ is the blow-up of the vertex, with exceptional divisor $E$, $f^*(K_X)\big|_E$ is not numerically trivial (thus, $f^*(K_X)\big|_E\nsim_\R0$). 

We will also show an example (example \ref{ex:cone/elliptic curve}) of a cone $X$ and a divisor $T'$ on $X$, such that, if $f:Y\rightarrow X$ is the blow-up of the vertex, with exceptional divisor $E$, $f^*(T')\big|_E\equiv0$ but $f^*(T')\big|_E\nsim_\R0$.

Finally, in section \ref{sect:lastexample}, we will give an example of a blow-up of a vertex of a cone $f:Y\rightarrow X$, with $Y$ smooth, where $K^-_{m,Y/X}>-1$ (so as to suggest that $X$ has log terminal singularities), but where for any choice of compatible boundary $\Delta$ on $X$ (with $\Delta$ a cone), $f$ is not a resolution of the pair $(X,\Delta)$. In particular, we can conclude that $X$ has \ltp singularities, but we cannot conclude that it has log terminal singularities (in the sense of \cite{dFH}).

In section 5, we will interpret the above results, proving a new result, and reproving, with direct methods, a result of \cite{BdFF}. Namely, we prove
\begin{thm}[Theorem \ref{thm:pullback and restriction, cones}] Let $X$ be a (projective) cone over a smooth projective variety, and let us assume that $X$ is normal; let $f:Y\rightarrow X$ be the blow-up of the vertex, and let $E$ be the exceptional divisor of $f$. Let $D$ be a Weil divisor on $X$. Then $f^*(D)\big|_E\sim_\R0$ if and only if $D$ is $\Q$-Cartier.
\end{thm}
and
\begin{cor}[Corollary \ref{cor:pullback and restriction, cones}, \cite{BdFF}, 2.29] Let $X$ be a (projective) cone over a smooth projective polarized variety $(V,L)$, and let us assume that $X$ is normal; let $f:Y\rightarrow X$ be the blow-up of the vertex, with exceptional divisor $E$. Let $D$ be a Weil divisor on $X$. The following are equivalent
\begin{enumerate}
\item $f^*(-D)=-f^*(D)$ (as $\R$-Weil divisors);
\item $f^*(D)\big|_E\equiv0$;
\item $D\sim_\Q C_\Delta$ is the cone over a divisor $\Delta$ on $V$ such that $\Delta\equiv rL$, for some $r\in\Q$.
\end{enumerate}
In particular, in any of the above cases, $f^*(D)$ is a $\Q$-divisor.
\end{cor}
\subsection*{Acknowledgements}
First of all, I would like to thank my advisor, S\'andor Kov\'acs for his help and support. A special thanks goes to Stefano Urbinati, for introducing me to his project of a ``non-$\Q$-Gorenstein MMP without boundaries''. I also would like to thank Tommaso de Fernex for his constant disponibility and interest. I would like to thank Christopher Hacon, Marco Andreatta, and Cinzia Casagrande, who suggested me to look at \cite{AW} for examples.

\section{Some general results}
\subsection{First definitions}
All the definitions and results in this part are of \cite{dFH}.

Unless otherwise stated, all varieties are normal over $\C$ and all divisors are Weil divisors. At times, when no confusion is likely, we will use the expression \emph{divisor} for an $\R$-Weil divisor, that is, $\R$-linear combinations of prime divisors.

Let $X$ be a normal variety. A \emph{divisorial valuation} on $X$ is a discrete valuation of the function field $k(X)$ of $X$ of the form $\nu=q\val_F$ where $q\in\R_{>0}$ and $F$ is a prime divisor over $X$, that is a prime divisor on some normal variety birational to $X$. Let $\nu$ be a discrete valuation. If $\Idlsh$ is a \emph{coherent fractional ideal} of $X$ (that is, a finitely generated sub-$\regsh_X$-module of the constant field $\mathcal{K}_X$ of rational functions on $X$), we set
$$
\nu(\Idlsh):=\min\{\nu(f)\,|\,\textrm{$f\in\Idlsh(U)$, $U\cap c_X(\nu)\neq\emptyset$}\}.
$$
\begin{df} To a given fractional ideal $\Idlsh$ we can associate a divisor, called the \emph{divisorial part} of $\Idlsh$, as
$$
\div(\Idlsh):=\sum_{E\subset X}\val_E(\Idlsh)E,
$$
where the sum runs over all the prime divisors on $X$; equivalently, $\div(\Idlsh)$ is such that
$$
\regsh_X(-\div(\Idlsh))=\Idlsh\dual\dual.
$$
\end{df}
\begin{df} Let $\nu$ be a divisorial valuation on $X$. The \emph{$\natural$-valuation} or \emph{natural valuation} along $\nu$ of a divisor $F$ on $X$ is
$$
\nu^\natural(F):=\nu(\regsh_X(-F)).
$$
\end{df}
De Fernex and Hacon show (\cite{dFH}, 2.8) that, for every divisor $D$ on $X$ and every $m\in\Z_{>0}$, $m\nu^\natural(D)\geq\nu^\natural(mD)$ and
$$
\inf_{k\geq1}\frac{\nu^\natural(kD)}{k}=\liminf_{k\rightarrow\infty}\frac{\nu^\natural(kD)}{k}=\lim_{k\rightarrow\infty}\frac{\nu^\natural(k!D)}{k!}\in\R.
$$
\begin{df} Let $D$ be a divisor on $X$ and $\nu$ a divisorial valuation. The \emph{valuation along $\nu$} of $D$ is defined to be the above limit
$$
\nu(D):=\lim_{k\rightarrow\infty}\frac{\nu^\natural(k!D)}{k!}.
$$
\end{df}
\begin{df} Let $f:Y\rightarrow X$ be a proper birational morphism from a normal variety $Y$. For any divisor $D$ on $X$, the \emph{$\natural$-pullback} of $D$ along $f$ is
\begin{eqnarray}\label{eq:naturalpullbackdef}
f^\natural D:=\div(\regsh_X(-D)\cdot\regsh_Y);
\end{eqnarray}
equivalently, $f^\natural D$ is the divisor on $Y$ such that
\begin{eqnarray}\label{eq:naturalpullbackchar}
\regsh_Y(-f^\natural D)=(\regsh_X(-D)\cdot\regsh_Y)\dual\dual.
\end{eqnarray}
\end{df}
\begin{df} We define the \emph{pullback} of $D$ along $f$ as
$$
f^*D:=\sum_{E\subset Y}\val_E(D)E,
$$
where the sum runs over all prime divisors on $Y$. Equivalently,
$$
f^*D=\liminf_k\frac{f^\natural(kD)}{k}\textrm{ coefficient-wise}.
$$
\end{df}
\begin{rk}\label{rk:infimum of divisors} When we talk about an infimum of divisors, we have to be very careful if we are considering the divisors as Weil divisors, or as numerical (or as linear) classes. In the latter case, the infimum is not well defined. For example, on $X=\Pspace^2$, let $C_d$ be the divisor of a degree $d\geq1$ curve. For each $d$, $C_d\sim dC_1$. Then, $\inf \{C_d/d\}=0$ as Weil divisors: if a divisor appears in the sequence, is with coefficient $1/d$, and then does not appear anymore. If we had chosen $C_1\sim_\Q C_d/d$ as representatives of the terms of our sequence, the infimum would have been $C_1$, and clearly $C_1$ is not $\Q$-linearly equivalent (or even numerically equivalent) with $0$.
\end{rk}
We notice that the evaluation along $\nu$ and the pullback above defined agree with corresponding notions in the case that the divisor $D$ is $\Q$-Cartier.
\begin{prop}[\cite{dFH}, 2.4 and 2.10] Let $\nu$ be a divisorial valuation on $X$ and let $f:Y\rightarrow X$ be a birational morphism from a normal variety $Y$. Let $D$ be any divisor let $C$ be any $\R$-Cartier divisor, with $t\in\R_{>0}$ such that $tC$ is Cartier.
\begin{enumerate}
\item The definitions of $\nu(C)$ and $f^*(C)$ given above coincides with the usual ones. More precisely,
$$
\nu(C)=\frac{1}{t}\nu(tC)\quad and \quad f^*(C)=\frac{1}{t}f^*(tC).
$$
Moreover,
$$
\nu(tC)=\nu^\natural(tC)\quad and\quad f^\natural(tC)=f^*(tC).
$$
\item The pullback is almost linear, in the sense that
\begin{eqnarray}\label{eq:pullbacklinear}
f^\natural(D+tC)=f^\natural(D)+f^*(tC)\quad and\quad f^*(D+C)=f^*(D)+f^*(C).
\end{eqnarray}
\end{enumerate}
\end{prop}
We observe that when working with natural valuation and natural pullback, the above properties are no longer true for $\R$-Cartier divisors which are not Cartier, see \cite{dFH}, 2.3. For example it may happen that $2C$ is Cartier, but $\nu^\natural(2C)\neq2\nu^\natural(C)$.
\begin{lm}[\cite{dFH}, 2.7]\label{lm:(fg)*-g*f*} Let $f:Y\rightarrow X$ and $g:V\rightarrow Y$ be two birational morphisms of normal varieties, and let $D$ be a divisor on $X$. The divisor $(fg)^\natural(D)-g^\natural f^\natural(D)$ is effective and $g$-exceptional. Moreover, if $\regsh_X(-D)\cdot\regsh_Y$ is an invertible sheaf, $(fg)^\natural(D)=g^\natural f^\natural(D)$.
\end{lm}
\subsection{Canonical divisors}
Given the above definitions for pullback of Weil divisors, there are different choices for relative canonical divisors, and hence for singularities. If $f:Y\rightarrow X$ is a proper birational morphism (of normal varieties) and if we choose a canonical divisor $K_X$ on $X$, we will always assume that the canonical divisor $K_Y$ on $Y$ be chosen such that $f_*K_Y=K_X$ (as Weil divisors).

Almost all the notions in this section are of \cite{dFH}, but will use the slightly different notation for simplicity in the statements. This notation was introduced by the author and Stefano Urbinati in \cite{ltsings}, 3.1.

\begin{df}[\cite{dFH}] Let $f:Y\rightarrow X$ be a proper birational map of normal varieties. The \emph{$m$-limiting relative canonical $\Q$-divisors} are
\begin{eqnarray*}
&&K^-_{m,Y/X}:=K_Y-\frac{1}{m}f^\natural(mK_X),\quad K^-_{Y/X}:=K_Y-f^*(K_X)\\
&&K^+_{m,Y/X}:=K_Y+\frac{1}{m}f^\natural(-mK_X),\quad K^+_{Y/X}:=K_Y+f^*(-K_X).
\end{eqnarray*}
\end{df}
As shown by \cite{dFH} (and as from the definitions), for all $m,q\geq1$,
\begin{eqnarray}\label{eq:K-<=K^+}
K^-_{m,Y/X}\leq K^-_{qm,Y/X}\leq K^-_{Y/X}\leq K^+_{Y/X}\leq K^+_{mq,Y/X}\leq K^+_{m,Y/X}
\end{eqnarray}
and
\begin{eqnarray}\label{K-(Y/X)=limsupK-(m,Y/X)}
K^-_{Y/X}=\limsup K^-_{m,Y/X}, \quad K^+_{Y/X}=\liminf K^+_{m,Y/X}
\end{eqnarray}
(coefficient-wise).
\begin{df} Let $\Delta$ be a boundary on $X$ and let $f:Y\rightarrow X$ be a proper birational morphism. The \emph{log relative canonical $\Q$-divisor} of $(Y,f_*^{-1}\Delta)$ over $(X,\Delta)$ is given by
$$
K^\Delta_{Y/X}:=K_Y+f_*^{-1}\Delta-f^*(K_X+\Delta)=K_Y+f_*^{-1}\Delta+f^*(-K_X-\Delta)
$$
where $f^{-1}_*\Delta$ is the strict transform of $\Delta$ on $Y$.
\end{df}

\begin{rk}\label{rk:dFH 3.9} With the same computation as \cite{dFH}, 3.9, we find that, if $\Delta$ is a boundary for $X$ and $m\geq1$ is such that $m(K_X+\Delta)$ is Cartier,
$$
\begin{array}{ll}
K^-_{m,Y/X}=K^\Delta_{Y/X}-\frac{1}{m}f^\natural(-m\Delta)-f_*^{-1}\Delta,\hspace{0.7em}&K^-_{Y/X}=K^\Delta_{Y/X}-f^*(-\Delta)-f_*^{-1}\Delta,\\
K^+_{m,Y/X}=K^\Delta_{Y/X}+\frac{1}{m}f^\natural(m\Delta)-f_*^{-1}\Delta,&K^+_{Y/X}=K^\Delta_{Y/X}+f^*(\Delta)-f_*^{-1}\Delta.
\end{array}
$$
Notice that, if $m$ is such that $m(K_X+\Delta)$ is Cartier and $\Delta$ is effective, we have
$$
K^\Delta_{Y/X}\leq K^-_{m,Y/X}\leq K^-_{Y/X}\leq K^+_{Y/X}\leq K^+_{m,Y/X}.
$$
\end{rk}
\begin{df} Let $Y\rightarrow X$ be a proper birational morphism with $Y$ normal, and let $F$ be a prime divisor on $Y$. For each integer $m\geq1$, the \emph{$m$-limiting discrepancy} of $F$ with respect to $X$ is
$$
a_m(F,X):=\ord_F(K^-_{m,Y/X}).
$$
The \emph{discrepancy} of $F$ with respect to $X$ is
$$
a(F,X):=\ord_F(K^-_{Y/X}).
$$
\end{df}
We recall that, from \cite{komo}, if $\Delta$ is a boundary for $K_X$, we have
$$
a(F,X,\Delta):=\ord_F(K^\Delta_{Y/X}).
$$
The next definition is of the author and Stefano Urbinati in \cite{ltsings}.
\begin{df}[\cite{ltsings}, 4.3]\label{df:sing} The variety $X$ is said to be satsfy condition $\mathrm{M}_{\geq-1}$ (resp. $\mathrm{M}_{>-1}$, resp. $\mathrm{M}_{\geq0}$, resp. $\mathrm{M}_{>0}$) if there is an integer $m_0$ such that $a_m(F,X)\geq-1$ (resp. $>-1$, resp. $\geq0$, resp. $>0$) for every prime divisor $F$ over $X$ and $m=m_0$ (and hence for any positive multiple $m$ of $m_0$).
\end{df}
\begin{thm}[\cite{dFH}, 7.2, \cite{ltsings}, 4.5]\label{thm:dFH 7.2} A variety $X$ satsfies condition $\mathrm{M}_{\geq-1}$ (resp. $\mathrm{M}_{>-1}$, resp. $\mathrm{M}_{\geq0}$, resp. $\mathrm{M}_{>0}$) if and only if there is an effective boundary $\Delta$ such that $(X,\Delta)$ is log canonical (resp. Kawamata log terminal, resp. canonical, resp. terminal).
\end{thm}
\begin{proof} In \cite{dFH}, the result is proven only for the conditions $\mathrm{M}_{\geq-1}$ and $\mathrm{M}_{>-1}$, but the same proof \emph{verbatim} works for the other two cases as well.
\end{proof}
\begin{df}[\cite{dFH}] A variety $X$ is said to have \emph{log terminal singularities} if it satisfies condition $\mathrm{M}_{>-1}$.
\end{df}
\subsection{Log terminal${}^+$ singularities}
The definitions and results in this section are of the author and Stefano Urbinati.
\begin{df}[\cite{ltsings}, 5.1] Let $Y\rightarrow X$ be a proper birational morphism with $Y$ normal, and let $F$ be a prime divisor on $Y$. The \emph{discrepancy${}^+$} of $F$ with respect to $X$ is
$$
a^+(F,X):=\ord_F(K^+_{Y/X}).
$$
\end{df}
We recall that in \cite{dFH}, a normal variety $X$ is defined \emph{canonical} (resp. \emph{terminal}), if $a^+(F,X)\geq0$ (resp. $a^+(F,X)>0$) for all prime divisors $F$, exceptional over $X$.
\begin{df}[\cite{ltsings}, 5.2] Let $X$ be a normal variety. we say that $X$ has \emph{log terminal${}^+$}, or simply, \emph{\ltp}, singularities if $a^+(F,X)>-1$ for all prime divisors $F$, exceptional over $X$.
\end{df}
\begin{lm}[\cite{ltsings}, 5.4]\label{lm:ltwithoneresolution} Let $X$ be a normal variety, and let $f:Y\rightarrow X$ be a log resolution, i.e. the exceptional locus of $f$ is a simple normal crossing divisor. If $a^+(E,X)>-1+\e$ for all prime exceptional divisors $E$ on $Y$, for some $\e>0$, then $a^+(F,X)>-1+\e$ for all prime divisors $F$ exceptional over $X$. In particular, then, $X$ is \ltp.
\end{lm}
Notice that the hypothesis of this lemma is satisfied if $a^+(E,X)>-1$ for all prime exceptional divisors $E$ on $Y$, since there are only finitely many such divisors.
\begin{prop}[\cite{ltsings}, 5.6]\label{prop:sheaf criterion} Let $X$ be a normal variety. Then $X$ is \ltp if and only if there exists $\e\in\Q$, $\e>0$, such that, for all sufficiently divisible $m\geq1$, and for all resolutions of $X$,
$$
\regsh_X(mK_X)\cdot\regsh_Y\subseteq\regsh_Y\big(m(K_Y+(1-\e)F_Y)\big),
$$
where $F_Y$ is the reduced exceptional divisor of $f$.
\end{prop}
\begin{rk}\label{rk:proof of dFH, 8.2} The proof of this proposition follows the proof of \cite{dFH}, 8.2, and uses the equivalence of the above inclusion of sheaves with the condition
$$
mK_Y+f^\natural(-mK_X)>m(-1+\e).
$$
\end{rk}
\begin{thm}[\cite{ltsings}, 5.10]\label{thm:fgcr} If $X$ is a \ltp normal variety, then $\Rscr(X,K_X):=\oplus_{m\geq0}\regsh_{X}(mK_X)$ is finitely generated. In this case, moreover, $X$ is klt if and only if $\Rscr(X,-K_X):=\oplus_{m\geq0}\regsh_{X}(-mK_X)$ is finitely generated.
\end{thm}
\begin{ex}\label{ex:lt+,notCM} In \cite{stefano}, 4.1 the author gives an example of a variety having canonical (hence \ltp singularities) but not klt. The example is a cone over an embedding of $\Pspace^1\times\Escr$, where $\Escr$ is an elliptic curve.
\end{ex}
\subsection{Pullbacks and boundaries}
We start with the definition of \emph{compatible boundary} of \cite{dFH}.
\begin{df} Let $X$ be a normal variety, and let $f:Y\rightarrow X$ be a proper birational map, with $Y$ normal. Let $D$ be a Weil divisor on $X$. For each $m\geq2$ a divisor $\Delta_m$ is called a \emph{weak $m$-compatible $D$-boundary with respect to $f$} such that
\begin{enumerate}
\item $\Delta_m\geq0$;
\item $D+\Delta_m$ is $\Q$-Cartier;
\item $\lfloor\Delta_m\rfloor=0$ and $m\Delta_m$ is integral; and
\item $\frac{f^\natural(mD)}{m}=f^*(D+\Delta_m)-f^{-1}_*\Delta_m$.
\end{enumerate}
Moreover, if there exist a proper birational morphism $g:Z\rightarrow X$  and a $D$-boundary $\Delta_m$ as above such that
\begin{enumerate}
\setcounter{enumi}{4}
\item $g$ is a resolution of $X$, and
\item $\mathrm{exc}(g)\cup f^{-1}_*\Delta_m$ has simple normal crossing support,
\end{enumerate}
such divisor will be called an \emph{$m$-compatible $D$-boundary}.
\end{df}
\begin{rk} This definition is slightly different than the one in \cite{dFH}, 5.1 and \cite{ltsings}, 3.6.
\end{rk}
We have the following generalization of \cite{dFH}, 5.4.
\begin{lm}[\cite{dFH}, 5.4, \cite{ltsings}, 3.6]\label{lm:compatible D-boundary} For each proper birationational morphism $f:Y\rightarrow X$ between normal varieties, any $m\geq2$ and any divisor $D$ on $X$, there exists a weak $m$-compatible $D$-boundary with respect to $f$. Moreover, for any $m\geq2$, there are $m$-compatible $D$-boundaries (with respect to some resolution).
\end{lm}
In \cite{ltsings} this results is used to prove the following proposition.
\begin{cor}[\cite{ltsings}, 3.9]\label{cor:lt3.9} If $f:Y\rightarrow X$ is any proper birational morphism of normal varieties,
$$
K^+_{Y/X}=\inf_\Delta K^{-\Delta}_{Y/X},
$$
where the infimum is taken over all the weak $m$-compatible $(-K_X)$-boundaries with respect to $f$.
\end{cor}
More generally,
\begin{cor}\label{cor:pullback and compatible $D$-boundary} Let $X$ be a normal variety, and let $f:Y\rightarrow X$ be a proper birational map, with $Y$ normal. Let $D$ be a Weil divisor on $X$. Let
$$
\mathscr{D}_D=\{\Delta\geq0,\textrm{$D+\Delta$ is $\Q$-Cartier},\lfloor\Delta\rfloor=0\}.
$$
Then
\begin{eqnarray}
f^*(D)=\inf_{\Delta\in\mathscr{D}_D}(f^*(D+\Delta)-f^{-1}_*\Delta),
\end{eqnarray}
coefficient-wise.
\end{cor}
\begin{proof} By definition,
$$
f^*(D)=\inf_m \frac{f^\natural(mD)}{m},
$$
coefficient-wise. By lemma \ref{lm:compatible D-boundary}, the left hand side is bigger or equal than the right one, since
\begin{eqnarray}
f^*(D)=\inf_m \frac{f^\natural(mD)}{m}=\inf_m(f^*(D+\Delta_m)-f^{-1}_*\Delta_m).
\end{eqnarray}
To finish the proof, let $\Delta\in\mathscr{D}_D$ such that $m(D+\Delta)$ is Cartier. As in remark \ref{rk:dFH 3.9} for the particular case of the relative canonical divisor,
$$
\frac{f^\natural(mD)}{m}\leq f^*(D+\Delta)-f^{-1}_*\Delta.
$$
\end{proof}
\begin{cor}\label{cor:pullback and restriction} Let $f:Y\rightarrow X$ be a proper birational morphism between normal varieties, with $Y$ smooth. Let $D$ be a Weil divisor on $X$, and let $E$ be any normal component of the exceptional divisor which is also a fiber (if it exists). The restriction of $f^*(D)$ to $E$ is numerically antieffective.
\end{cor}
\begin{proof} By the above lemma \ref{lm:compatible D-boundary},
$$
f^*(D)\big|_E=\Big(\inf_m(f^*(D+\Delta_m)-f^{-1}_*\Delta_m)\Big)\Big|_E,
$$
where $\Delta_m$ are weak $m$-compatible $D$-boundaries, which are effective by construction. For each $m$,
$$
(f^*(D+\Delta_m)-f^{-1}_*\Delta_m)\big|_E\sim_\Q-f^{-1}_*\Delta_m\big|_E.
$$
Since $\Delta_m\geq0$ and $E$ is not contained in the support of $f^{-1}_*\Delta_m$, $-f^{-1}_*\Delta_m\big|_E\leq0$. Taking the limit, we get the desired result.
\end{proof}
\begin{rk} As discussed in (F\ref{it:F3}), we do not know that the restriction to a fiber of $f^*(D)$ will be $\R$-linearly equivalent to an antieffective divisor.
\end{rk}
Notice that we have the following example.
\begin{ex}[\cite{dFH}, 2.3] Let $X\subseteq\Aspace^3$ be the cone over a conic, $X=\{x^2+y^2=z^2\}$, let $f:Y\rightarrow X$ be the blow-up of the vertex, with exceptional divisor $E$. Let $L$ be a line passing thorugh the origin. Notice that $L$ is $\Q$-Cartier, but not Cartier. In this case $f^\natural(L)=f^{-1}_*L+E$, so that $f^\natural(L)\big|_E\sim_\Q-pt$, which is not zero (even though $L$ is $\Q$-Cartier).
\end{ex}
\begin{rk}\label{rk:positivity on fibers} A word of caution is required. The behavior in general of the restriction of the pullback of a Weil divisor to a fiber is extremely unclear. Even assuming that we are working in the case of a resolution (restricting Weil divisors is not necessary possible), the positivity of the restriction is not a given. In the previous corollary \ref{cor:pullback and restriction}, for example, a key point is that $E$ being a divisor, it cannot be contained in the support of $f^{-1}_*\Delta_m$. In this particular case we see that the restriction of the fibers of the pullback is ``negative''. More evidence of this is given by the construction of the \emph{nef envelope} by \cite{BdFF} (we refer to the article for the definitions). Roughly speaking, to an $\R$-Weil divisor $D$, they associate a $b$-divisor $\mathrm{Env}_X(D)$, which is nef over $X$ and whose determination on a model $\pi:X_\pi\rightarrow X$ corresponds to $-\pi^\natural(-D)$. However, this ``negativity'' of the pullback does not happen when we consider the restriction to fibers which are not themselves divisors. Let $D$ be a Weil divisor on $X$ such that $\Rscr(X,D)$ is finitely generated, and let $f:Y=\boldsymbol{\mathrm{Proj}}_X\Rscr(X,D)\rightarrow X$. This is the situation of flipping contractions, for example. In this case $f$ is small and $f^{-1}_*(D)=f^\natural(D)=f^*(D)$ is $\Q$-Cartier and $f$-ample. Thus, for any curve $C$ in the fiber of $f$, $f^*(D).C>0$.
\end{rk}

There is another generalization of \cite{dFH}, 5.4. This is very similar to \cite{trans}, 2.9, but there is done for $K_X$ and $K_S$.
\begin{lm}\label{lm:relativecompatibleboundary} Let $X$ be a normal variety and let $S\subseteq X$ be a normal (effective) Cartier divisor in $X$, and let $f:Y\rightarrow X$ be a proper birational morphism. Then there exists a weak $m$-compatible $(-K_X)$-boundary $\Delta$ on $X$ such that $\Delta\big|_S$ is a weak $m$-compatible $(-K_S)$-boundary. Moreover, $f$ and $\Delta$ can be chosen so that $\Delta$ and $\Delta\big|_S$ are $m$-compatible.
\end{lm}
\begin{proof} Let us fix some notation. Let $T:=f^{-1}_*S$ and let $g:=f\big|_T:T\rightarrow S$, which is still proper birational.

Let $D$ be an effective divisor such that $-K_X-D$ is Cartier. By adjunction, $(-K_X-S-D)\big|_S=-K_S-D\big|_S$ is still Cartier. Let $\Lscr$ be an line bundle such that $\Lscr\otimes\regsh_X(-mD)$ and $\Lscr\big|_S\otimes\regsh_S(-mD\big|_S)$ are generated by global sections, and let $G$ be a general element in the linear system $\{H\in|\Lscr|\,, H-mD\geq0\}$, which we can assume reduced and having no common components with $D$ and $S$. Let $M=G-mD$ and
$$
\Delta_m=\frac{M}{m}.
$$
As in the proofs of \cite{dFH}, 5.4 and \cite{trans}, 2.9, the generality of the choice of $G$ guarantees that $\Delta_m$ is a weak $m$-compatible $(-K_X)$-boundary with respect to $f$, and that $\Delta_m\big|_S$ is a weak $m$-compatible $(-K_S)$-boundary with respect to $g$.

If $f$ is a resolution of $\regsh_X(mK_X)+\regsh_X(-mD)$ such that $g:T\rightarrow S$ is a resolution of $\regsh_S(mK_S)+\regsh_S(-mD\big|_S)$, then $\Delta_m$ can be chosen general enough so that $\exc(f)\cup f^{-1}_*\Delta_m$ and $\exc(g)\cup g^{-1}_*(\Delta_m\big|_S)$ have simple normal crossing support.
\end{proof}

\section{Some properties of \ltp singularities}
\subsection{Multiplier ideal}
We will show the existence of a multiplier ideal which detects \ltp singularities.
\begin{df}\label{df:multideal} Let $X$ be a normal variety, and let $f:Y\rightarrow X$ be a log resolution of $X$. We can define the ideal
\begin{eqnarray}
\Jscr^+(X):=f_*\regsh_Y\big(\lceil K_Y+f^*(-K_X)\rceil\big).
\end{eqnarray}
\end{df}
\begin{lm}\label{lm:multideal} The above ideal is independent of the resolution.
\end{lm}
\begin{proof} If is enough to prove this result for two log resolutions $f:Y\rightarrow X$ and $h:W\rightarrow X$ with $h=f\circ g$, for some $g:W\rightarrow Y$. It is known that $K^+_{W/X}-g^*K^+_{Y/X}$ is effective and $g$-exceptional \cite{dFH}, 2.13; thus, the result follows as in the usual setting of multiplier ideals.

Let $h^*(-K_X)-g^*f^*(-K_X)=E^+_g$, which is effective and $g$-exceptional \cite{dFH}, 2.13. We have the equalities
\begin{eqnarray*}
h_*\regsh_W\big(\lceil K_W+h^*(-K_X)\rceil\big)&=&h_*\regsh_W\big( K_{W/Y}+g^*(K_Y)+\lceil h^*(-K_X)\rceil\big)=\\
&=&f_*g_*\regsh_W\big(K_{W/Y}+\lceil h^*(-K_X)\rceil+g^*(K_Y)\big)=\\
&=&f_*\Big(g_*\regsh_W\big(K_{W/Y}+\lceil h^*(-K_X)\rceil\big)\otimes\regsh_Y(K_Y)\Big)
\end{eqnarray*}
by the projection formula. We are done if we prove that
$$
g_*\regsh_W\big(K_{W/Y}+\lceil h^*(-K_X)\rceil\big)=\regsh_Y(\lceil f^*(-K_X)\rceil).
$$
Notice that we reduced to considering a map between smooth varieties and $\supp f^*(-K_X)$ is snc. Hence,
\begin{eqnarray*}
g_*\regsh_W\big(K_{W/Y}+\lceil h^*(-K_X)\rceil\big)&=&g_*\regsh_W\big(K_{W/Y}+\lceil g^*f^*(-K_X)+E^+_g\rceil\big)=\\
&=&g_*\regsh_W\big(K_{W/Y}+\lceil g^*f^*(-K_X)\rceil+F^+_g\big),
\end{eqnarray*}
where $F^+_g$ is effective and $g$-exceptional. Indeed, since $E^+_g$ is effective and $g$-exceptional, $\lceil g^*f^*(-K_X)+E^+_g\rceil\geq\lceil g^*f^*(-K_X)\rceil$ and the difference is $g$-exceptional. Thus,
\begin{eqnarray*}
g_*\regsh_W\big(K_{W/Y}+\lceil h^*(-K_X)\rceil\big)&=&g_*\regsh_W\big(K_{W/Y}+\lceil g^*f^*(-K_X)\rceil+F^+_g\big)=\\
&=&g_*\regsh_W\big(K_{W/Y}+F^+_g-\lfloor- g^*f^*(-K_X)\rfloor\big)=\\
&=&g_*\regsh_W\big(K_{W/Y}+F^+_g-\lfloor g^*(-f^*(-K_X))\rfloor\big)=\\
&=&\regsh_Y(-\lfloor -f^*(-K_X)\rfloor)=\\
&=&\regsh_Y(\lceil f^*(-K_X)\rceil)
\end{eqnarray*}
by \cite{laz}, II.9.2.19.
\end{proof}
\begin{cor}\label{cor:multideal} A normal variety $X$ has \ltp singularities if and only if $\Jscr^+(X)=\regsh_X$.
\end{cor}
\begin{rk}\label{rk:multideal} As shown in lemma \ref{lm:compatible D-boundary}, for each $m\geq2$, we can find an anti-effective compatible boundary $\Delta_m$ on a log resolution $f:Y\rightarrow X$ of $(X,\Delta_m)$ such that $K^+_{m,Y/X}=K^{\Delta_m}_{Y/X}$. So, a priori, we could construct multiplier fractional ideals for each $m$ for the pair $(X,\Delta_m)$ and then do a limiting process on these multiplier ideals, as it was done in \cite{dFH}, 5.5, to construct multiplier ideals for the log terminal singularities. However, there are two issues with this approach. The first one is that, since $K^+_{mq,Y/X}\leq K^+_{m,Y/X}$, such limiting multiplier ideal should be a minimal element, and not a maximal one (in the collection of multiplier ideals for each $m$). Hence, there is no guarantee that it exists. The second problem is that, even when such minimal ideal exists, it may not detect the singularities, as example \ref{ex:Km>-1,K=-1} shows. Indeed, in that case, for each $m$ the multiplier ideal associated to $(X,\Delta_m)$ would be $\regsh_X$ since, for any log resolution $Y\rightarrow X$, $K^+_{m,Y/X}>-1$. Indeed, any log resolution would factor through the blow-up of the vertex of the cone $Z\rightarrow X$ and $g:Y\rightarrow Z$, and $K^+_{m,Y/X}\geq K_{Y/Z}+g^*(K^+_{m,Z/X})$. However, $\Jscr^+(X)\neq\regsh_X$ since $X$ does not have \ltp singularities.
\end{rk}

\subsection{Bertini type theorem}
\begin{thm}\label{thm:bertini} Let $X$ be a normal variety having \ltp singularities. Then the generic hyperplane section of $X$ has \ltp singularities.
\end{thm}
\begin{proof} Let $f:Y\rightarrow X$ be any log resolution of $X$. Then, for a generic hyperplane section $S$ of $X$, $S$ will be normal and $T:=f^{-1}_*S=f^*S$ will be smooth. Moreover, the map $g:=f\big|_T:T\rightarrow S$ will be a log resolution. By lemma \ref{lm:relativecompatibleboundary}, for each $m\geq 2$, we can find $(-K_X)$-boundaries $\Delta_m$ on $X$ such that
\begin{enumerate}
\item $-K_X+\Delta_m$ and $-K_S+\Delta_{m,S}$ are $\Q$-Cartier, where $\Delta_{m,S}:=\Delta_m\big|_S$,
\item $K^+_{m,Y/X}=f^*(-K_X+\Delta_m)-f^{-1}_*\Delta_m=K_{Y/X}^{-\Delta_m}$, and
\item $K^+_{m,T/S}=g^*(-K_S+\Delta_{m,S})-g^{-1}_*\Delta_{m,S}=K_{T/S}^{-\Delta_{m,S}}$.
\end{enumerate}
As in \cite{singsofpairs}, 7.7 (for example), $\discrep(X,-\Delta_m)\leq\discrep(S,-\Delta_{m,S})$. Therefore,
$$
-1<\inf_m\discrep(X,-\Delta_m)\leq\inf_m\discrep(S,-\Delta_{m,S}).
$$
Notice that, since $f$ and $g$ are log resolution, the above discrepancies are computed directly by looking at the orders of $K_{Y/X}^{-\Delta_m}$ and $K_{T/S}^{-\Delta_{m,S}}$ along exceptional divisors over $X$ and $S$ respectively. As
$$
K^+_{T/S}=\inf_m K^+_{m,T/S}=\inf_m K_{T/S}^{-\Delta_{m,S}},
$$
$S$ has \ltp singularities (proposition \ref{lm:ltwithoneresolution}).
\end{proof}
Similarly, we can prove the following.
\begin{prop} Let $X$ be a normal variety having canonical (resp. terminal) singularities. Then, the generic hyperplane section of $X$ has canonical (resp. terminal) singularities.
\end{prop}
\begin{cor} Let $X$ be a normal variety having terminal singularities; then $X$ is regular in codimension $2$. Let $X$ be a normal variety having \ltp singularities; then $X$ is $\Q$-factorial in codimension $2$.
\end{cor}
\begin{proof} By taking hyperplane sections (and using induction on dimension), we reduce to the case $\dim X=2$. Let $f:Y\rightarrow X$ be any proper birational map. For surfaces the pullback corresponds to the numerical pullback (\cite{BdFF}, 2.20) which is linear, and thus $K^-_{Y/X}=K^+_{Y/X}$. But then, as in \cite{dFH}, 7.13, if $X$ is terminal (resp. \ltp) this is equivalent to satisfying condition $M_{>0}$ (resp. $M_{>-1}$). Therefore $X$ is smooth (resp. $\Q$-factorial).
\end{proof}

\subsection{Inversion of adjunction and small deformations}
\begin{thm}\label{thm:invofadjunction} Let $X$ be a normal variety, and let $S$ be an effective (normal and reduced) Cartier divisor in $X$ having \ltp singularities. Then $X$ has \ltp singularities in a neighborhood of $S$.
\end{thm}
As in the usual inversion of adjunction for klt singularities, \cite{komo}, 5.50, this result relies on the following connectedness theorem.
\begin{thm}[\cite{komo}, 5.48]\label{thm:connectedthm} Let $f:Y\rightarrow X$ be a proper and birational morphism, $Y$ smooth, $X$ normal. Let $D=\sum d_i D_i$ be an snc $\Q$-divisor on $Y$ such that $f_*D$ is effective and $-(K_Y+D)$ is $f$-nef. Write $F=\sum_{i:d_i\geq1} D_i$; then $\supp F$ is connected.
\end{thm}
Now we will use this theorem to prove the above inversion of adjunction. The first result is a version of \cite{komo}, 5.50 for pairs with an anti-effective boundary.
\begin{lm}\label{lm:komo5.50} Let $X$ be a normal variety, and let $S$ be an effective (normal and reduced) Cartier divisor in $X$. Let $f:Y\rightarrow X$ be a log resolution of $(X,S)$, which restricts to a log resolution $g:=f\big|_T:T\rightarrow S$, $T:=f^{-1}_*S$. Let $\Delta$ be an effective divisor such that $K_X-\Delta$ is $\Q$-Cartier, and let $\Delta_S=\Delta\big|_S$. If the discrepancy of $(S,-\Delta_S)$ relative to $g$ is bigger than $-1$, then so is the one of $(X,-\Delta+S)$ with respect to $f$ in a neighborhood of $S$.
\end{lm}
\begin{proof} Let $\Delta^Y=f^{-1}_*\Delta$. Let us write
$$
K_Y-\Delta^Y+T+F=f^*(K_X+S-\Delta)+A,
$$
where all the coefficient in $F$ are bigger or equal than $-1$, and all the coefficients of $A$ are strictly bigger than $-1$. Moreover, $T+F-A$ is snc, and $F$ is effective. Restricting the above identity to $T$, we obtain
$$
K_T-\Delta_S^T=g^*(K_S-\Delta_S)+(A-F)\big|_T.
$$
The discrepancy of $(S,-\Delta_S)$ (with respect to $g$) is bigger than $-1$ if and only if $F\cap T=\emptyset$, while the discrepancy of $(X,-S+\Delta)$ (with respect to $f$) is bigger than $-1$ if and only if $F\cap f^{-1}(S)=\emptyset$. Let $D=T+F-A$; notice that $f_*D=S\geq0$ and
$$
-(K_Y+D)=\Delta^Y+f^*(K_X+S-\Delta)
$$
is $f$-nef. By theorem \ref{thm:connectedthm}, each $x\in S$ has an open neighborhood $x\in U_x\subseteq X$ such that $(T\cup F)\cap f^{-1}(U_x)$ is connected, hence $F\cap f^{-1}(U_x)=\emptyset$. Moving $x\in S$, we obtain the claim.
\end{proof}
\begin{rk} Notice that in the statement we do not assume that the maps $f$ and $g$ are resolutions of the pairs $(X,-\Delta+S)$ and $(S,-\Delta_S)$, so that the discrepancies relative to $f$ and $g$ are not necessarily the discrepancies of the pairs (i.e., non necessarily those pairs are log terminal).
\end{rk}
\begin{proof}[Proof of theorem \ref{thm:invofadjunction}] Let us fix a log resolution $f:Y\rightarrow X$ of $(X,S)$, which restricts to a log resolution $g:=f\big|_T:T\rightarrow S$, $T:=f^{-1}_*S$. For each $m\geq2$, we can find a weak $m$-compatible $(-K_X)$-boundary $\Delta_m$ for $f$ such that $\Delta_{m,S}:=\Delta_m\big|_S$ is a weak $m$-compatible $(-K_S)$-boundary for $g$ (lemma \ref{lm:relativecompatibleboundary}). As $K^-_{T/S}\leq K^{-\Delta_m}_{T/S}$ and $S$ has \ltp singularities, the discrepancy of $(S,-\Delta_{m,S})$ with respect to $g$ is bigger than $-1$. By lemma \ref{lm:komo5.50}, the discrepancy of $(X,-\Delta_m+S)$ with respect to $f$ is bigger than $-1$. As $S$ is Cartier and effective, the discrepancy of $(X,-\Delta_m)$ with respect to $f$ is bigger than $-1$. More precisely, we have
\begin{eqnarray*}
K_Y-\Delta_m^Y+T&=&f^*(K_X-\Delta_m+S)+\sum a_iE_i=\\
&=&f^*(K_X-\Delta_m)+f^*(S)+\sum a_iE_i=\\
&=&f^*(K_X-\Delta_m)+T+\sum b_iE_i+\sum a_iE_i=\\
&=&f^*(K_X-\Delta_m)+T+\sum(a_i+b_i)E_i,
\end{eqnarray*}
where $f^*(S)-T=\sum b_iE_i\geq0$, so that
$$
K_Y-\Delta_m^Y=f^*(K_X-\Delta_m)+\sum(a_i+b_i)E_i\geq f^*(K_X-\Delta_m)+\sum a_iE_i.
$$
As in the proof of lemma \ref{lm:komo5.50}, in this case the coefficients of the discrepancy of $(S,-\Delta_{m,S})$ (with respect to $g$) are exactly the ones of the discrepancy of $(X,-\Delta_m+S)$ in a neighborhood of $S$ (with respect to $f$). As the
$$
\inf_m K^{-\Delta_{m,S}}_{T/S}=K^+_{T/S}
$$
has all coefficient strictly bigger than $-1$, so has
$$
\inf_m K^{-\Delta_m}_{Y/X}=K^+_{Y/X}
$$
(in a neighborhood of $S$). Thus $X$ has \ltp singularities in a neighborhood of $S$ (proposition \ref{lm:ltwithoneresolution}).
\end{proof}
As immediate corollary of theorem \ref{thm:invofadjunction} we have deformation invariance.
\begin{cor}\label{cor:definv} Let $f:X\rightarrow T$ be a proper flat family of varieties over a smooth curve $T$ and $t\in T$ a closed point. If the fiber $X_t$ has \ltp singularities, then so do the other fibers near $t$.
\end{cor}
\begin{proof} The technique of this proof is standard (see, for example, \cite{ks}, 4.2). By theorem \ref{thm:invofadjunction}, $X$ has \ltp singularities near $X_t$. Let $Z$ be the non-\ltp locus of $X$, which is closed in $X$. Since $f$ is proper, $f(Z)$ is a closed subset of $T$ not containing $t$. By restricting $T$ to an open set, we may assume that $X$ has \ltp singularities. By theorem \ref{thm:bertini}, all fibers over nearby points of $t\in T$ have \ltp singularities.
\end{proof}

\subsection{Rational singularities}
We start by recalling the definition of rational singularities.
\begin{df} Let $X$ be a variety. $X$ is said to have \emph{rational singularities} if for a resolution $f:Y\rightarrow X$ (or, equivalently, for all resolutions) $Rf_*\regsh_Y\simeq_{q.is.}\regsh_X$.
\end{df}
We have the following useful characterization of rational singularities (see for example \cite{singsofpairs}, 11.9)
\begin{thm}[Kempf's criterion] Let $X$ be a normal variety and $f:Y\rightarrow X$ a resolution. Then $X$ has rational singularities if and only if $f_*\omega_Y=\omega_X$ and $X$ is Cohen-Macaulay. 
\end{thm}
The next result is in the line of various results, which relate log terminal singularities and rational singularities.
\begin{thm}\label{thm:lt++CM=>rtl} If $X$ has Cohen-Macaulay and \ltp singularities, then it has rational singularities.
\end{thm}
\begin{proof} Let $Y$ be any resolution. For each prime divisor $E$ in $Y$ exceptional over $X$,
$$
-1<\ord_E(K^+_{Y/X})\leq\ord_E\big(K_Y+f^\natural(-K_X)\big).
$$
Since $K_Y+f^\natural(-K_X)$ is an integral divisor, we have that
$$
K_Y+f^\natural(-K_X)\geq0.
$$
As in the proof of \cite{dFH}, 8.2 (see remark \ref{rk:proof of dFH, 8.2}), this is equivalent with the condition $\regsh_X(K_X)\cdot\regsh_Y\subseteq\regsh_Y(K_Y)$, that is
$$
\omega_X\cdot\regsh_Y\subseteq\omega_Y.
$$
Pushing forward on $X$ this inclusion ($f_*$ is left exact) we obtain
$$
f_*(\omega_X\cdot\regsh_Y)\subseteq f_*\omega_Y.
$$

Notice that, if $\Fscr$ is a torsion-free sheaf on a normal variety $X$ subsheaf of $\Kscr_X$ and $f:Y\rightarrow X$ is any proper birational morphism, then
$$
\Fscr\subseteq f_*(\Fscr\cdot\regsh_Y).
$$
Indeed $\Fscr\cdot\regsh_Y$ is the image of $f^*\Fscr$ in $\Kscr_Y$. We have natural maps
$$
\Fscr\rightarrow f_*f^*\Fscr\rightarrow f_*(\Fscr\cdot\regsh_Y).
$$
The composition above is an isomorphism at the generic point of $X$, and thus the kernel must be torsion. As $\Fscr$ is torsion-free, the above composition is injective. In our case, we have the natural inclusion $\omega_X\subseteq f_*(\omega_X\cdot\regsh_Y)$.

We have the chain of inclusions
$$
\omega_X\subseteq f_*(\omega_X\cdot\regsh_Y)\subseteq f_*\omega_Y\subseteq\omega_X,
$$
which proves
$$
f_*\omega_Y=\omega_X.
$$
As $X$ is Cohen-Macaulay, $X$ has rational singularities.
\end{proof}
\begin{rk} As previously observed, this is just one of the several results relating log terminal singularities and rational singularities. The first one is Elkik's theorem, \cite{elk}, which says that, $\Q$-Gorenstein log terminal singularities are rational. More generally, \cite{komo}, 5.22, dlt singularities are rational. We recall that also log terminal singularities in the sense of \cite{dFH} are rational (\cite{dFH}, 7.7). 

Of particular interest is also the following result of de Fernex and Docampo \'Alvarez (\cite{dFDA}, 7.2). Let $X$ be a (normal) variety such that $rK_X$ is Cartier, and let $\dfrak_X$ denote the lci defect ideal of $X$. Then if $X$ has rational singularities, $(X,\dfrak_X^{-1})$ is $J$-canonical. The converse holds if $X$ is Cohen-Macaulay. We refer to the paper of de Fernex and Docampo \'Alvarez for all the relevant definitions.
\end{rk}
\begin{rk} The assumption of Cohen-Macaulay is necessary. The example \ref{ex:lt+,notCM} has \ltp singularities but it is not Cohen-Macaulay. This phenomenon is similar to the one happening in the case of Jacobian singularities (see previous remark). In the result \cite{dFDA}, 7.2, the assumption of Cohen-Macaulay for the converse direction is essential as well.
\end{rk}

\section{Some examples}
In this section, we limit ourselves to some examples; the theoretical discussion will follow in the next section. For the first example, we will write the entire computation. The other examples use the same ideas.
\subsection{First example}
\begin{ex}\label{con:example} Let us consider the $(1:2)$ embedding of $\Pspace^1\times\Pspace^1$ in $\Pspace^5$; let us call it $Z$. Let $X'=C(Z)\subseteq\Aspace^6$ be the affine cone over it, and let $X=\overline{C(Z)}\subseteq\Pspace^6$ be the projective cone. Notice that the singularities of affine and projective cones are the same.

Let us fix some notation. Let $f_1=\Pspace^1\times\{pt\}$ and $f_2=\{pt\}\times\Pspace^1$ and let $L\equiv f_1+2f_2$ be the ample divisor giving the embedding. The divisor group of $\Pspace^1\times\Pspace^1$ is generated by $f_1,f_2$ and we have $f_1.f_2=1$, $f_1^2=f_2^2=0$.

Let $f:Y\rightarrow X$ be the blow-up of the vertex of the cone, which is a log resolution of $X$, and let $E$ be the exceptional divisor.
\end{ex}
\begin{lm} The cone $X$ in \ref{con:example} is normal and non-$\Q$-Gorenstein.
\end{lm}

\begin{prop}\label{prop:limiting can sings} The cone $X$ in \ref{con:example} has $K^-_{Y/X}=0$. Moreover, for all $m\geq1$, $K^-_{m,Y/X}=0$. In particular, $X$ has canonical singularities.
\end{prop}
\begin{lm}[\cite{stefano}, proof of 3.3]\label{lm:stefano3.3.3} Let $V$ be a smooth variety, and let $L$ be a very ample line bundle on $V$ determining an embedding $V\hookrightarrow\Pspace^n$, and let $X$ be the projective cone over it. Let us assume that $X$ is normal. Let $f:Y\rightarrow X$ be the blow-up of the vertex, with exceptional divisor $E$. Then
\begin{eqnarray}\label{eq:K-Y/X}
\ord_E(K^-_{Y/X})=-(1+t),
\end{eqnarray}
where
$$
t=\inf\{s\in\Q\,|\,\textrm{$sL-K_V\sim_\Q\Delta$, $\Delta$ is effective}\}.
$$
\end{lm}
\begin{proof} This is a particular case of claim \ref{claim:pullback of divisor}, for $D=K_V-L$, as $K_X\sim C_{K_V}-C_L$.
\end{proof}
\begin{proof}[Proof of \ref{prop:limiting can sings}] Notice that, since $Z\cong\Pspace^1\times\Pspace^1$, $K_Z\sim-2f_1-2f_2$. For each $s\in\R$,
$$
sL-K_Z\sim s(f_1+2f_2)-(-2f_1-2f_2)=(s+2)f_1+(2s+2)f_2,
$$
which is effective as long as $s\geq-1$. Therefore, with the notation of the lemma, $t=-1$ and $\ord_E(K^-_{Y/X})=-(1-1)=0$. Thus $K^-_{Y/X}=0$, and $X$ has limiting canonical singularities.

Notice that, for all $m\geq 1$, and each $\Delta$ such that $m(K_X+\Delta)$ is Cartier,
$$
K^\Delta_{Y/X}\leq K^-_{m,Y/X}\leq K^-_{Y/X}=0.
$$
If $\Delta=C_{f_1}$, the cone over $f_1$, then $K_Z+f_1\sim-2f_1-2f_2+f_1=-f_1-2f_2=-L$. By \cite{kol}, 70.(1), $K_X+\Delta$ is Cartier, and therefore
$$
K^\Delta_{Y/X}\leq K^-_{m,Y/X},
$$
for all $m$. As discussed in the proof of the lemma, we have that $f_1=(-1)L-K_Z$ and $K^\Delta_{Y/X}=-(1-1)E=0$. Thus, for all $m\geq1$, we have the chain of inequalities
$$
0=K^\Delta_{Y/X}\leq K^-_{m,Y/X}\leq K^-_{Y/X}=0,
$$
which concludes the computation.
\end{proof}
\begin{cor}\label{cor:can sings example 1} The cone in \ref{con:example} satisfies the condition $M_{\geq0}$, that is, there exists an effective boundary on $X$ such that $(X,\Delta)$ is canonical.
\end{cor}
\begin{proof} The basis of the cone is $\Pspace^1\times\Pspace^1$, and each effective divisor on $\Pspace^1\times\Pspace^1$ is $\Q$-linearly equivalent to a smooth one.
\end{proof}
\begin{cor}\label{cor:pullback and restriction for example} With the notation of \ref{con:example}, $f^*(K_X)\big|_E\sim_\Q-f_1$, which is neither $\R$-linearly equivalent nor numerically equivalent to $0$.
\end{cor}
\begin{proof} We have that $K^-_{Y/X}=0$, so that $K_Y=f^*(K_X)$. Then,
$$
(K_Y+E)\big|_E\sim_\Q K_E\sim-2f_1-2f_2
$$
($E\cong Z\cong\Pspace^1\times\Pspace^1$) and
$$
(f^*(K_X)+E)\big|_E\sim_\Q f^*(K_X)\big|_E+E\big|_E\sim_\Q f^*(K_X)\big|_E-L\sim f^*(K_X)\big|_E-f_1-2f_2.
$$
Therefore, $f^*(K_X)\big|_E\sim_\Q-f_1$.
\end{proof}
\subsection{Three more examples}
With the techniques of the previous example, we can study three more examples. The computations are the same as before, and thus left to the reader.
\begin{ex}\label{ex:Km>-1,K=-1} Let $\Lscr=\regsh_{\Pspace^1}(3)$, and let $\pi:V=\relSpec_{\Pspace^1}S(\Lscr)\rightarrow\Pspace^1$. Then $V$ is an elliptic surface, and $K_V\sim f$, where $f$ denotes the fiber \cite{mir}, III.1.1. Since the sheaf of (local) sections of $\pi$ is $\Lscr$ \cite{har}, Ex II.5.18(2), which has global sections, there exists a section $s:\Pspace^1\rightarrow V$. Then $\Pic(V)/\Pic^0(V)\cong NS(V)$, $\Pic^0(V)\cong\Pic^0(\Pspace^1)$, $\Pic(V)/\Pic(C)\cong NS(V)/\Z\overline{f}$ \cite{mir}, VII.1.1 and VII.1.3. Moreover, if $V_\eta$ is the general fiber of $\pi$, $\Pic(V)\rightarrow\Pic(V_\eta)$ is surjective, with kernel generated by classes of vertical divisors \cite{mir}, VII.1.5. Notice that, in particular, since $\Pic(\Pspace^1)\cong\Z$, the kernel of the restriction $\Pic(V)\rightarrow\Pic(V_\eta)$ is generated by the class of $f$.

If we choose $L$ on $V$ to be $L\sim f+\sum e_iP_i$, where $\sum e_iP_i$, $e_i>0$, is ample on the elliptic curve $V_\eta$, $L$ will be ample on $V$. Then, as long as $s\geq0$, $sL+K_V\sim_\Q D\geq0$. However, as soon as $s<0$, $sL+K_V\sim_\Q (s+1)f-\sum e_i P_i$ which cannot be linearly equivalent to an effective divisor. Thus,
$$
K^+_{Y/X}=-1.
$$
On the other hand, with similar computations to the ones in \cite{zhang}, 4.9, $K^+_{m,Y/X}>-1$ for all $m$.
\end{ex}

\begin{ex}\label{ex:cone/elliptic curve} Let $(\mathscr{E},O)$ be an elliptic curve with an embedding in $\Pspace^n$ such that the cone over it is normal, and let $X$ be the projective cone over $\mathscr{E}$. Let $P$ be a point on $\mathscr{E}$ with infinite order, and let $T=P-O$. Then $T$ is non-torsion, but $T\in\Pic^0(\mathscr{E})$, $T\equiv0$. Let $f:Y\rightarrow X$ be the blow-up of the vertex of the cone, with exceptional divisor $E$. Let $C_T$ be the cone over $T$ in $X$. The computation in claim \ref{claim:pullback of torsion} shows that $f^*C_T=f^{-1}_*C_T$, so that $f^*C_T\big|_E=T\equiv0$ but $f^*(C_T)\big|_E\nsim_\R0$.

The cone over an elliptic curve is not log terminal, however, we can use this idea to give an example of such a phenomenon on an \ltp variety. For example, if $X$ is the cone over $\Escr\times\Pspace^1$ in \cite{stefano}, 4.1 (see example \ref{ex:lt+,notCM}), then we can consider the cone over $T'=T\boxtimes\regsh_{\Pspace^1}$, where $T\in\Pic^0(\Escr)$. If $f:Y\rightarrow X$ is the blowup of the vertex of $X$, with exceptional divisor $E$, we have an example of a divisor $C_{T'}$ on a variety with \ltp singularities, such that $f^*C_{T'}\big|_E=T'\equiv0$ but $f^*(C_{T'})\big|_E\nsim_\R0$.
\end{ex}
\begin{ex} The formula in claim \ref{claim:pullback of divisor} gives us a nice way of creating examples. Let $V$ be the $(2:3)$ embedding of $\Pspace^1\times\Pspace^1$ in $\Pspace^n$, and let $X$ be the projective cone over it. It can be checked that $X$ is normal. The above mentioned formula gives
$$
K^-_{Y/X}=-\frac{1}{3}E,\quad K^+_{Y/X}=0,
$$
giving an example of a singularity log terminal and canonical (in the sense of \cite{dFH}), but with $K^-_{Y/X}$ less then $0$, hence not admitting an effective boundary $\Delta$ such that $(X,\Delta)$ is canonical (in the usual sense).
\end{ex}
\subsection{Last example}\label{sect:lastexample}
Finally, we will show an example of a cone over a smooth variety where by simply looking at boundaries that are cones and at the blow-up of the vertex, we cannot determine the singularity. We remark that the technique used to compute $K^{-}_{Y/X}$ in this example was used by \cite{stefano}, and in the first example of this section, and in the example right above.

First we will identify the hypothesis that we need.
\begin{ex}\label{ex:Fano} Let $V$ be a variety, $\Delta$ be a divisor on $V$, and $n$ a positive integer such that
\begin{enumerate}
\item $V$ is smooth Fano;
\item $\Delta$ is effective and, for each $m\geq0$, all the divisors in the linear system $|m\Delta|$ have singularities worst than normal crossing;
\item $-nK_V-\Delta$ is very ample and, for each $\e>0$,
$$
\big|-(\frac{1}{n}+\e)\Delta+n\e K_V\big|=\emptyset.
$$
\end{enumerate}
\end{ex}
As mentioned in the introduction, finding such a Fano variety in low dimension seems a hard task. We could not find any such examples in dimension 3 or less. The example we found is in dimension 4. We will discuss this later. Now we will show why the above hypothesis are useful.
\begin{prop}\label{ex:cone/Fano} Let $V$, $\Delta$ and $n$ as in \ref{ex:Fano}, and let $L:=-nK_V-\Delta$, which is very ample. Let $X$ be the cone over $(V,L)$, and let $f:Y\rightarrow X$ be the blow-up of the vertex. Then $K^-_{Y/X}>-1$, and for $1||m$, for each $\Delta_m$ such that $\Delta_m$ is a cone over a divisor in $V$ and $K^-_{m,Y/X}=K^{\Delta_m}_{Y/X}$, $f$ is not a resolution of the pair $(X,\Delta_m)$.
\end{prop}
\begin{cor} If $X$ is as in the above proposition, although $K^-_{Y/X}>-1$, we can not say what are the singularities of the variety $X$.
\end{cor}
\begin{proof} Notice that condition (a) is used to guarantee the existence of a positive $n$ such that $-nK_V-\Delta$ is very ample.

As in \ref{lm:stefano3.3.3}, $K^-_{Y/X}=(-1-t)E$, where $t=\inf\{s\,|\,sL-K_V\sim_\Q\Gamma\geq0\}$. Condition (c) implies that $t=-\frac{1}{n}$, so that $K^-_{Y/X}=(-1+\frac{1}{n})E>-E$.

If we consider $s=-\frac{1}{n}$, then $sL-K_V=\frac{1}{n}\Delta$, so that, for $1||m$, $m(K_V+\frac{1}{n}\Delta)$ is Cartier, and the cone over $\Delta$ is a boundary. Thus, for $1||m$, $K^-_{m,Y/X}=K^-_{Y/X}$. Let $C_{\Gamma_m}$ be an effective boundary computing the discrepancy, and a cone over a divisor $\Gamma_m$ in $V$. The divisor $\Gamma_m$ will have to be $\Q$-linearly equivalent to $\Delta$ (or to some multiple), and thus a singular non-nc divisor, by condition (b). But then the blow-up of the vertex is not a log resolution of the pair $(X,C_{\Gamma_m})$.
\end{proof}

\vspace{1ex}

Now we will discuss of how to realize such a Fano variety (and such $\Delta$). We will give two examples. The first one is a simplified version of the second one, which is of \cite{AW}.
\begin{ex} Let $V_1\rightarrow\Pspace^4$ be the blow-up of a point. Let $S_0\cong\Pspace^3$ be the exceptional divisor. Notice that $V_1$ has a natural structure of $\Pspace^1$-bundle over $\Pspace^3$, $V_1\cong\Pspace(\regsh\oplus\regsh(1))$. Let $H$ be the pullback on $V_1$ of an hyperplane section in $\Pspace^3$. Then $L=2H+S_0$ is base point free. Let $B\in|2L|$ be a smooth general member, and let $V\rightarrow V_1$ be the double cover ramified over $B$. There is a natural map $V\rightarrow\Pspace^4$, which admits Stein factorization $V\rightarrow Z\rightarrow\Pspace^4$. The first map in the Stein factorization is a divisorial contraction, while the second map is generically $2:1$. Let $E$ be the exceptional divisor of $V\rightarrow Z$. 

It is just a computation to check that $V$ is a smooth Fano variety. Explicitely, $K_V$ is the pullback of $-3H-2S_0$.

We claim that we can choose a smooth $B$ such that $E$ is singular. Indeed, $E$ is the double cover of $S_0$ ramified over $B\cap S_0$. The divisor $B\cap S_0\in|\regsh_{\Pspace^3}(2)|$, and we can choose it to be a degree $2$ cone of equations $x_0^2+x_1^2=x_2^2$. By \cite{AW}, 3.5.2, this can be extended to a smooth divisor $B$ on $V$. By direct computation, as in \cite{EV}, 3.13, we see that $E$ is still singular (with equation $t^2+x_2^2=x_0^2+x_1^2$).

Since $E$ is an exceptional divisor, for any $m\geq0$, $|mE|=\{mE\}$.

Condition of (c) is the last thing to be verified. This again can be directly checked. We have that $L=-nK_V-E$ is very ample for $n\geq2$. For the last condition of (c), the key point to notice that, if $D\sim_\Q-(\frac{1}{n}+\e)E+n\e K_V$, then $f_*D=-\alpha S_0-\beta H$, with $\alpha,\beta>0$, and thus $f_*D$ can not be effective. But then $D$ cannot be effective.

In conclusion, $V$ and $\Delta:=E$ satisfy the conditions of \ref{ex:Fano}.
\end{ex}
\begin{ex}[\cite{AW}, 3.5.4] Let $V_1\rightarrow\Pspace^3$ be the blow-up of $\Pspace^3$ at one point, and let $S_0$ be the exceptional divisor of the blow-up. The variety $V_1$ has a $\Pspace^1$-bundle structure $V_1=\Pspace(\regsh\oplus\regsh(1))\rightarrow\Pspace^2$; let $H$ be the pullback via this map of the line. Let $Y$ be the product $V_1\times\Pspace^1$, with projections $p_1$ and $p_2$. Let $L:=p_1^*(S_0+2H)$, and let $B$ be a smooth divisor in $|2L|$. As above, we can consider the $2:1$ cover $V\rightarrow Y$, ramified over $B$, and consequently the Stein factorization $V\rightarrow Z\rightarrow\Pspace^3\times\Pspace^1$. The morphism $V\rightarrow Z$ is a divisorial contraction of a divisor $E$ which is mapped to $\Pspace^1$. From the theory of $3$-dimensional good contractions, \cite{mori}, we know that the generic fiber of $E\rightarrow\Pspace^1$ is either a smooth quadric or a quadric cone. In the latter case, with a similar computation to the one above, it can be checked that the pair $(V,\Delta:=E)$ satisfies the conditions of \ref{ex:Fano}.
\end{ex}

\section{The meaning of restrictions}
\begin{thm}\label{thm:pullback and restriction, cones} Let $X$ be a (projective) cone over a smooth projective variety, and let us assume that $X$ is normal; let $f:Y\rightarrow X$ be the blow-up of the vertex, and let $E$ be the exceptional divisor of $f$. Let $D$ be a Weil divisor on $X$. Then $f^*(D)\big|_E\sim_\R0$ if and only if $D$ is $\Q$-Cartier.
\end{thm}
\begin{proof} Let us fix the notation. Let $V$ be a smooth variety, let $L$ be a very ample line bundle on $V$ giving an embedding $V\hookrightarrow\Pspace^n$, and let $X\subseteq\Pspace^{n+1}$ be the projective cone over it. We are assuming that $X$ is normal. We have that $V\subseteq X$ as section at infinity.

Let $D$ be a Weil divisor in $X$. Since $X$ is a projective cone, by \cite{har}, exercise II.6.3, $D\sim C_{D|_V}$ the cone over the restriction of $D$ to $V$. The pullback of \cite{dFH} coincides with the usual pullback for Cartier divisors, thus it preserves linear equivalence. Hence $f^*(D)\sim_\Q f^*(C_{D|_V})$, and $f^*(D)\big|_E\sim_\R f^*(C_{D|_V})\big|_E$. Since the property of being $\R$-Cartier is also preserved by linear equivalence, it is enough to show that, if $D$ is a divisor on $V$, $f^*(C_D)\big|_E\sim_\R0$ if and only if $D\sim_\Q rL$, for some $r\in\Q$ (which is the only case in which $C_D$ can be $\Q$-Cartier, by \cite{kol}, 70.(1)).

\begin{claim}\label{claim:pullback of divisor} Let $D$ be a divisor on $V$ and let $C_D$ be the cone over it in $X$. Then
\begin{eqnarray}\label{eq:pullback of divisor}
f^*(C_D)=f^{-1}_*C_D+tE,
\end{eqnarray}
where
$$
t=\inf\{s\,|\,sL-D\sim_\Q\Delta,\lfloor\Delta\rfloor=0\}.
$$
\end{claim}
\begin{proof}[Proof of claim \ref{claim:pullback of divisor}.]
Let $D$ be a divisor on $V$ and let $C_D$ be the cone over it in $X$. Then
\begin{eqnarray*}
f^*(C_D)&=&\inf\big\{(f^*(C_D+\Gamma)-f^{-1}_*\Gamma)\,\big|\,\textrm{$C_D+\Gamma$ is $\Q$-Cartier, $\Gamma\geq0$, $\lfloor\Gamma\rfloor=0$}\big\}\leq\\
&\leq&\inf\big\{(f^*(C_D+C_\Delta)-f^{-1}_*C_\Delta)\,\big|\,\textrm{$C_D+C_\Delta$ is $\Q$-Cartier},\\
&&\hspace{20em}C_\Delta\geq0,\,\lfloor C_\Delta\rfloor=0\big\}.
\end{eqnarray*}
Let $C_\Delta$ be as above. Then $\Delta\geq0$ and $\lfloor\Delta\rfloor=0$. Since $C_D+C_\Delta$ is $\Q$-Cartier, there exists $s\in\Q$ such that $\Delta\sim_\Q sL-D$. If $\Delta\sim_\Q sL-D$, then
\begin{eqnarray*}
f^*(C_D+C_\Delta)-f^{-1}_*C_\Delta&\sim_\Q&f^*(C_{sL})-f^{-1}_*C_{sL-D}=f^{-1}_*C_{sL}+sE-f^{-1}_*C_{sL-D}=\\
&=&f^{-1}_*C_D+sE.
\end{eqnarray*}
Therefore,
$$
\val_E(f^*(C_D))\leq\inf\{s\,|\,sL-D\sim_\Q\Delta,\Delta\geq0\}.
$$
Let now $\Gamma$ be a divisor on $X$ with $\Gamma\geq0$, $C_D+\Gamma$ $\Q$-Cartier,  and
$$
f^*(C_D+\Gamma)-f^{-1}_*\Gamma=f^{-1}_*C_D+kE,
$$
with $k\in\Q$. Let $\Delta=f^{-1}_*\Gamma\big|_E$, which we can think of as a divisor on $V$, since $E\cong V$. Then, $\Delta\geq0$ and
$$
\Delta=f^{-1}_*\Gamma\big|_E\sim_\Q(-kE-f^{-1}_*C_D)\big|_E\sim_\Q kL-D,
$$
which means that $C_D+C_\Delta$ is $\Q$-Cartier. With the same computation of above,
$$
f^*(C_D+C_\Delta)-f^{-1}_*C_\Delta=f^{-1}_*C_D+kE=f^*(C_D+\Gamma)-f^{-1}_*\Gamma.
$$
Therefore, if
$$
t=\inf\{s\,|\,sL-D\sim_\Q\Delta,\lfloor\Delta\rfloor=0\},
$$
then
$$
\val_E(f^*(C_D))=t,
$$
which means that
$$
f^*(C_D)=f^{-1}_*C_D+tE.
$$
\end{proof}
Restricting the above identity to $E$, we find
\begin{eqnarray}\label{eq:pullback of divisor and restriction}
f^*(C_D)\big|_E\sim_\R-tL+D=-(tL-D),
\end{eqnarray}
which is $0$ if and only if $D\sim_\R tL$. Since $D$ and $L$ are a integral Weil divisors, it means that $t\in\Q$, and that the condition is equivalent with $C_D$ being $\Q$-Cartier.
\end{proof}
The next result is similar to \cite{BdFF}, 2.29.
\begin{cor}\label{cor:pullback and restriction, cones} Let $X$ be a (projective) cone over a smooth projective polarized variety $(V,L)$, and let us assume that $X$ is normal; let $f:Y\rightarrow X$ be the blow-up of the vertex, with exceptional divisor $E$. Let $D$ be a Weil divisor on $X$. The following are equivalent
\begin{enumerate}
\item $f^*(-D)=-f^*(D)$ (as $\R$-Weil divisors);
\item $f^*(D)\big|_E\equiv0$;
\item $D\sim_\Q C_\Delta$ is the cone over a divisor $\Delta$ on $V$ such that $\Delta\equiv rL$, for some $r\in\Q$.
\end{enumerate}
In particular, in any of the above cases, $f^*(D)$ is a $\Q$-divisor.
\end{cor}
\begin{proof} Notice that, if $\Delta=D\big|_V$, then $D\sim C_\Delta$ (by by \cite{har}, exercise II.6.3). Since all the conditions in the above statement are invariants under $\Q$-linear equivalence, we can directly assume that $D=C_\Delta$, with $\Delta$ a divisor on $V$.
\begin{itemize}
\item[\desc{(a)$\Rightarrow$(b)}] By corollary \ref{cor:pullback and restriction}, $f^*(D)\big|_E$ is numerically antieffective. On the other hand, $f^*(D)\big|_E\equiv -f^*(-D)\big|_E$ is numerically effective. Thus, it must be numerically trivial.
\item[\desc{(b)$\Rightarrow$(c)}] Let $D=C_\Delta$ and $f^*(D)=f^{-1}_*C_\Delta+tE$. Then, $0\equiv f^*(D)\big|_E\equiv\Delta-tL$. This means that $\Delta\equiv tL$. Since they are both divisors, $t\in\Q$.
\item[\desc{(c)$\Rightarrow$(a)}] Let $D=C_\Delta$ and $\Delta\equiv rL$, with $r\in\Q$. It is enough to prove that
\begin{eqnarray}\label{eq:numerical equivalence and pullback}
f^*(D)=f^{-1}_*D+rE
\end{eqnarray}
(the same $r$!). If this is true, then $-\Delta\equiv-rL$ and
$$
f^*(-D)=f^{-1}_*(-D)-rE=-f^{-1}_*D-rE=-(f^{-1}_*D+rE)=-f^*(D).
$$
Notice that the formula \eqref{eq:numerical equivalence and pullback} shows that, in this case, $f^*(D)$ is a $\Q$-divisor. In order to prove \eqref{eq:numerical equivalence and pullback}, we can show that, if $\Delta\equiv0$, then
$$
f^*(C_\Delta)=f^{-1}_*C_\Delta.
$$
Let $\Delta\equiv rL$ and $T=\Delta-rL\equiv 0$, and let us assume that $f^*(C_T)=f^{-1}_*C_T$, as claimed above. Then
$$
f^*(C_\Delta)=f^*(C_{T+rL})=f^*(C_T+C_{rL})=f^*(C_T)+f^*(C_{rL}),
$$
since $C_{rL}$ is $\Q$-Cartier. In turn,
$$
f^*(C_T)+f^*(C_{rL})=f^{-1}_*C_T+f^{-1}_*(C_{rL})+rL=f^{-1}_*C_\Delta+rL,
$$
as desired (the strict transform is linear). It remains to show the following claim.
\begin{claim}\label{claim:pullback of torsion} Let $T\equiv0$. Then $f^*(C_T)=f^{-1}_*C_T$. 
\end{claim}
\begin{proof}[Proof of claim \ref{claim:pullback of torsion}.] By claim \ref{claim:pullback of divisor}, $f^*(C_T)=f^{-1}_*(C_T)+tE$ with
$$
t=\inf\{s\,|\,sL-T\sim_\Q\Delta,\lfloor\Delta\rfloor=0\}.
$$
For each $s\in\Q$, $s>0$, since $L$ is very ample and $T\equiv 0$, $sL-T\equiv sL$ is ample. Therefore $sL-T\sim_Q\Delta$, $\Delta\geq0$. On the other hand, for each $s\in\Q$, $s<0$, $sL-T\equiv sL$ is antiample, and thus, $sL-T\sim_\Q\Delta$, $\Delta\leq0$. Therefore,
$$
t=\inf\{s\,|\,sL-T\sim_\Q\Delta,\Delta\geq=0\}=0.
$$
\end{proof}
\end{itemize}
\end{proof}
\begin{rk} The claim \ref{claim:pullback of torsion} does not give a necessary condition, as example \ref{con:example} shows. With the notation of that example, we have a divisor $T=-2f_2$ numerically non-trivial such that $f^*(C_T)=f^{-1}_*C_T$. This can be easily checked by direct computation.
\end{rk}
The previous corollary and \cite{BdFF}, 2.29 give
\begin{cor}\label{cor:pullback of numerically Cartier} Let $X$ be a (projective) cone over a smooth projective polarized variety $(V,L)$, and let us assume that $X$ is normal; let $f:Y\rightarrow X$ be the blow-up of the vertex. Let $D$ be a Weil divisor on $X$. Then $f^*(-D)=-f^*(D)$ if and only if $D$ is numerically Cartier, i.e., $g^*(-D)=-g^*(D)$ for any proper birational morphism of normal varieties $g:Z\rightarrow X$.
\end{cor}


\end{document}